\documentclass{amsart}

\usepackage{enumitem}
\usepackage{amsfonts, amssymb, amsthm, amsmath, calc, cancel, cite,color, eucal, graphics, graphicx, hyperref, latexsym, mathdots, multirow, pgfplots, theoremref, tikz,tikz-cd, url}
\usepackage{accents}

\numberwithin{equation}{section}

\theoremstyle{plain}
\newtheorem{theorem}{Theorem}[section]
\newtheorem{lemma}[theorem]{Lemma}
\newtheorem{prop}[theorem]{Proposition}
\newtheorem{cor}[theorem]{Corollary}

\theoremstyle{definition}
\newtheorem{definition}[theorem]{Definition}

\newtheorem{rem}[theorem]{Remark}

\newtheorem{con}[theorem]{Conjecture}

\renewcommand{\phi}{\varphi}
\renewcommand{\epsilon}{\varepsilon}
\newcommand{\mb}[1]{\mathbf{#1}}
\renewcommand{\vec}[1]{\mathbf{#1}}
\newcommand{\bb}[1]{\mathbb{#1}}

\newcommand{\mc}[1]{\mathcal{#1}}

\newcommand{\A}{\mb{A}}

\newcommand{\B}{\mb{B}}
\newcommand{\C}{\mb{C}}
\newcommand{\D}{\mb{D}}

\renewcommand{\P}{\mb{P}}

\newcommand{\M}{\mb{M}}
\newcommand{\n}{\mb{n}}
\newcommand{\U}{\mb{U}}

\newcommand{\Q}{\mb{Q}}

\newcommand{\x}{\vec{x}}

\newcommand{\0}{\mb{0}}
\newcommand{\1}{\mb{1}}
\newcommand{\2}{\mb{2}}

\newcommand{\PI}{\mb{3}^\text{p}}
\newcommand{\TC}{\mb{3}^\text{c}}

\newcommand{\bdot}{\boldsymbol{\cdot}}

\renewcommand{\emptyset}{\varnothing}
\newcommand{\meet}{\wedge}
\newcommand{\join}{\vee}
\DeclareMathOperator{\glb}{glb}
\DeclareMathOperator{\ds}{ds}
\DeclareMathOperator{\cd}{cd}
\DeclareMathOperator{\dnd}{dnd}
\DeclareMathOperator{\pid}{pid}
\DeclareMathOperator{\id}{id}
\DeclareMathOperator{\emd}{emd}

\title{Satisfiability degrees for BCK-algebras}

\title{Satisfiability degrees for BCK-algebras}
\author{C. Matthew Evans}
\date{}

\begin{document}

\maketitle

\begin{abstract}
We investigate the satisfiability degree of some equations in finite BCK-algebras; that is, given a finite BCK-algebra and an equation in the language of BCK-algebras, what is the probability that elements chosen uniformly randomly with replacement satisfy that equation?

Specifically we consider the equations for the excluded middle, double negation, commutativity, positive implicativity, and implicativity. We give a sufficient condition for an equation to have a finite satisfiability gap among commutative BCK-algebras, and prove that the law of the excluded middle has a gap of $\frac{1}{3}$, while the positive implicative and implicative equations have gap $\frac{1}{9}$. More generally, though, in the language of BCK-algebras, we show that double negation, commutativity, positive implicativity, and implicativity all fail to have finite satisfiability gap. We provide bounds for the probabilities in these cases.
\end{abstract}

\section{Introduction}

A fascinating result of Gustafson \cite{gustafson73} states that in a finite non-Abelian group $G$, the probability that two elements commute is at most $\frac{5}{8}$. There is a substantial literature about commuting probabilities in finite groups; we refer the reader to the survey \cite{DNP13} by Das, Nath, and Pournaki, and the bibliography contained therein. One can generalize the commuting probability for groups in two natural directions:
\begin{enumerate}
\item What can we say about commuting probabilities in other algebraic systems?
\item What is the probability that an algebraic system satisfies some specified first-order formula?
\end{enumerate} 

For an example of (1), MacHale \cite{machale76} investigated commuting probability for finite rings. While less has been written about commuting probabilities in the case of rings, there is some recent work in \cite{BM13}, \cite{BMS13}, and \cite{BDN17}.

There are several examples of (2) in the case of groups. Probabilities for some commutator-like equations were considered in \cite{DJM20} and \cite{lescot95}, while \cite{kocsis20} studies probabilities for some equations generalizing the commutativity equation.

A recent paper of Bumpus and Kocsis \cite{BK22} generalizes in both directions by considering probability questions for the class of Heyting algebras, which are the algebraic semantics for intuitionistic logic. Inspired by their work, the present paper considers probabilities for some equations in the language of BCK-algebras, a class of algebraic structures introduced by Imai and Is\'{e}ki \cite{II66} which are the algebraic semantics for a non-classical logic having only implication. These probabilities are formalized as follows:

\begin{definition} Given a first-order language $\mathcal{L}$, a finite $\mathcal{L}$-structure $M$, and an $\mathcal{L}$-formula $\phi(x_1, x_2, \ldots x_n)$ in $n$ free variables, the quantity
\[\ds(\phi, M)=\frac{|\,\{\,(a_1, a_2,\ldots, a_n)\in M^n\mid \phi(a_1, a_2, \ldots a_n)\,\}|}{|M|^n}\] is the \emph{degree of satisfiability} of the formula $\phi$ in the structure $M$.
\end{definition}

\begin{definition} Let $T$ be a theory over a first-order language $\mathcal{L}$ and $\phi$ an $\mathcal{L}$-formula in $n$ free variables. We say that $\phi$ has \emph{finite satisfiability gap} $\epsilon$ in $T$ if there is a constant $\epsilon >0$ such that, for every finite model $M$ of $T$, either $\ds(\phi, M)=1$ or $\ds(\phi, M)\leq 1-\epsilon$\,.
\end{definition}

Taking Gustasfon's result as an example in the language of groups, the equation $xy=yx$ has finite satisfiability gap $\frac{3}{8}$: every finite group either has degree of satisfiability 1 (if it is Abelian) or no larger than $\frac{5}{8}$.

The paper is structured as follows: in the next section we define the class of BCK-algebras and introduce the equations we will consider throughout the paper. 

In section 3, we give sufficient conditions for an equation in the language of commutative BCK-algebras to have a finite satisfiability gap. In section 4 we consider two different equations in one variable: the law of the excluded middle and double negation. We show the law of the excluded middle has finite satisfiability gap using the result of section 3, but also show that double negation does not have finite satisfiability gap among bounded BCK-algebras. We find algebras realizing bounds on the probability for double negation. In section 5 we consider three different equations in two variables: commutativity, positive implicative, and implicative. None of these have a finite satisfiability gap in the language of BCK-algebras, and we gives algebras realizing bounds for these probabilities. However, the latter two equations do have finite satisfiability gap among commutative BCK-algebras.

Provided here is a summary of the various bounds discussed in the paper. Let $\A$ be a BCK-algebra of order $n$. Let $\cd(\A)$ denote the commuting degree, $\dnd(\A)$ the double negation degree, $\pid(\A)$ the positive implicative degree, and $\id(\A)$ the implicative degree.
\begin{itemize}
\item If $\A$ is non-commutative, then \[\frac{3n-2}{n^2}\leq \cd(\A)\leq \frac{n^2-2}{n^2}\,.\]
\item If $\A$ is non-commutative and bounded, then \[\frac{2}{n}\leq \dnd(\A)\leq \frac{n-1}{n}\,.\]
\item If $\A$ is not positive implicative, then \[\frac{4n-4}{n^2}\leq \pid(\A)\leq \frac{n^2-1}{n^2}\,.\]
\item If $\A$ is not implicative, then \[\frac{4n-4}{n^2}\leq \id(\A)\leq \frac{n^2-1}{n^2}\,.\]
\item If $\A$ is linear but not positive implicative, then \[\frac{n^2+3n-2}{2n^2}\leq \pid(\A)\leq \frac{n^2-1}{n^2}\,.\]
\item If $\A$ is linear but not implicative, then \[\frac{n^2+3n-2}{2n^2}\leq \id(\A)\leq \frac{n^2-1}{n^2}\,.\]
\end{itemize}

\section{Preliminaries}

\begin{definition}A \emph{BCK-algebra} is an algebra $\A=\langle A; \bdot, 0\rangle$ of type $(2,0)$ such that 
\begin{enumerate}
\item[(BCK1)] $\bigl[(x\bdot y)\bdot(x\bdot z)\bigr]\bdot(z\bdot y)=0$
\item[(BCK2)] $\bigl[x\bdot (x\bdot y)\bigr]\bdot y=0$
\item[(BCK3)] $x\bdot x=0$
\item[(BCK4)] $0\bdot x=0$
\item[(BCK5)] $x\bdot y=0$ and $y\bdot x=0$ imply $x=y$.
\end{enumerate} for all $x,y,z\in A$.
\end{definition}

These algebras are partially ordered by: $x\leq y$ if and only if $x\bdot y=0$. Note that 0 is the least element by (BCK4). One can also show that $x\bdot 0=x$ for all $x\in \A$. If the Hasse diagram for $\A$ is a chain, we will say $\A$ is \emph{linear}. A \emph{bounded BCK-algebra} is an algebra $\A=\langle A; \bdot, 0, 1\rangle$ of type $(2,0,0)$ such that $\langle A; \bdot, 0\rangle$ is a BCK-algebra and $x\bdot 1=0$ for all $x\in A$.

Define a term operation $\meet$ by $x\meet y:=y\bdot(y\bdot x)$. The element $x\meet y$ is a lower bound for $x$ and $y$, but in general is not the greatest lower bound of $x$ and $y$. As it turns out, $x\meet y =\glb\{x,y\}$  if and only if $x\meet y=y\meet x$. If this equation holds for all $x,y\in\A$, we say $\A$ is \emph{commutative}.

If $\A$ is bounded, define the term operation $\neg x:= 1\bdot x$. While we always have $\neg\neg\neg x=\neg x$, in general $\neg\neg x\leq x$. That is, we cannot always eliminate a double negative in a bounded BCK-algebra.

Finally, if $\A$ is bounded and commutative, define the term operation \[x\join y :=\neg(\neg x\meet \neg y)\,.\] We note that the term-reduct $\langle A; \meet, \join\rangle$ is a distributive lattice \cite{traczyk79}.

In this paper we consider the following equations:
\begin{enumerate}
\item[(DN)] $\neg\neg x=x$
\item[(EM)] $x\join\neg x=1$
\item[(T)] $x\meet y=y\meet x$
\item[(E$_1$)] $x\bdot y=(x\bdot y)\bdot y$
\item[(I)] $x\bdot(y\bdot x)=x$
\end{enumerate} If $\A$ satisfies (E$_1$) we say it is \emph{positive implicative}, and if $\A$ satisfies (I) we say it is \emph{implicative}. The classes of commutative, positive implicative, and implicative BCK-algebras are important subclasses of BCK-algebras; in fact, each of these subclasses is a variety. Further, a BCK-algebra is implicative if and only if it is both commutative and positive implicative.  For more detail on the elementary properties of BCK-algebras -- particularly the term operations $\neg$, $\meet$, and $\join$ -- we refer the reader to \cite{it76, it78, mj94}.

We now give methods for constructing new BCK-algebras which will be used throughout the paper. Let $\{\A_\lambda\}_{\lambda\in\Lambda}$ be a family of BCK-algebras such that $A_\lambda\cap A_\mu=\{0\}$ for $\lambda\neq \mu$, and let $U$ denote the union of the $A_\lambda$'s. Equip $U$ with the operation 
\[x\bdot y=\begin{cases}x\bdot_\lambda y &\text{if $x,y\in A_\lambda$}\\x&\text{otherwise}\end{cases}\,.\] Then $U$ is a BCK-algebra we will call the \emph{BCK-union} of the $\A_\lambda$'s and denote it by $\U=\bigsqcup_{\lambda\in\Lambda}\A_\lambda$. This construction first appears in \cite{it76}, though only in the case $|\Lambda|=2$. For a full proof that this is a BCK-algebra, see \cite{evans20}.

Next, given any BCK-algebra $\A$ of order $n-1$, we construct a new BCK-algebra of order $n$ by appending a new top element, call it $\top$, and extending the BCK-operation as follows:
\begin{align*}
x\bdot \top&=0\\
\top\bdot \top&=0\\
\top\bdot x&=\top
\end{align*} for all $x\in\A$. This is known as \emph{Is\'{e}ki's extension of $\A$}, which we will denote $\A\oplus\top$. Is\'{e}ki's extension always yields a bounded non-commutative BCK-algebra having $\A$ as a maximal ideal \cite{iseki75}.

There are three BCK-algebras of small order that will be used several times throughout the paper. We give their Cayley tables in Table \ref{tab:algs}.
\begin{table}[h]
{\centering
\begin{tabular}{c||cc}
$\bdot$  & 0              & 1 \\\hline\hline
0             & 0              & 0 \\
1             & 1              & 0
\end{tabular}
\hspace{1cm}
\begin{tabular}{c||ccc}
$\bdot$  & 0              & $1$ & $2$ \\\hline\hline
0             & 0              & 0 & 0 \\
$1$             & $1$              & 0 & 0 \\
$2$             & $2$              & $2$ & 0 
\end{tabular}
\hspace{1cm}
\begin{tabular}{c||ccc}
$\bdot$  & 0              & $1$ & $2$ \\\hline\hline
0             & 0              & 0 & 0 \\
$1$             & $1$              & 0 & 0 \\
$2$             & $2$              & $1$ & 0 
\end{tabular}
\caption{\label{tab:algs} The algebras $\2$, $\PI$, and $\TC$}
}
\end{table} The algebra $\2$ is the unique (up to isomorphism) BCK-algebra of order 2, and it is implicative. The algebra $\PI$ is positive implicative but not commutative. The algebra $\TC$ is commutative but not positive implicative. All three are linear.

We will also need the following family of algebras $\C_n$ (for $n\geq 2$) defined as follows: the carrier set is $C_n=\{0, 1,2, ,\ldots, n-1\}$ and the operation is \[x\bdot y=\max\{x-y, 0\}\,.\] These algebras are known to be linear, commutative BCK-algebras. Note that $\C_2=\2$ and $\C_3=\TC$.

Finally, we mention an important result of Romanowska and Traczyk.

\begin{theorem}[\cite{RT80}]\label{RT80}
Every finite commutative BCK-algebra is a product of BCK-chains.
\end{theorem}

\section{Sufficient conditions for finite satisfiability gap}

In this section we give sufficient conditions for an equation to have a finite satisfiability gap in the language of commutative BCK-algebras. Suppose $\phi(x_1, \ldots, x_k)$ is an equation in $k$ free variables in the language of commutative BCK-algebras. For a finite BCK-algebra $\A$, let \[S(\A)=\{(a_1,\ldots ,a_k)\in \A^k\,\mid\, \phi^\A(a_1,\ldots, a_k)\}\] so that \[\ds(\phi, \A)=\frac{|S(\A)|}{|\A|^k}\,.\]

\begin{lemma}\label{ds is mult}
Degree of satisfiability is multiplicative; that is, $\ds(\phi, \A\times\B)=\ds(\phi,\A)\cdot\ds(\phi,\B)$.
\end{lemma}

\begin{proof}
Note that 
\begin{align*}
\bigl((a_1,b_1), \ldots, (a_k,b_k)\bigr)\in S(\A\times\B) &\Longleftrightarrow \phi^{\A\times\B}\bigl((a_1,b_1),\ldots, (a_k,b_k)\bigr)\\
&\Longleftrightarrow \bigl(\phi^\A(a_1,\ldots, a_k)\,,\,\phi^\B(b_1,\ldots, b_k)\bigr)\\\
&\Longleftrightarrow (a_1,\ldots ,a_k)\in S(\A) \text{ and } (b_1,\ldots, b_k)\in S(\B)\\ 
&\Longleftrightarrow \bigl((a_1,\ldots, a_k)\,,\,(b_1,\ldots, b_k)\bigr)\in S(\A)\times S(\B)\,
\end{align*} which tells us $|S(\A\times\B)|=|S(\A)\times S(\B)|=|S(\A)||S(\B)|$. Therefore, 
\[\ds(\phi,\A\times\B)=\frac{|S(\A\times\B)|}{|\A\times\B|^k}=\frac{|S(\A)||S(\B)|}{|\A|^k|\B|^k}=\ds(\phi,\A)\cdot\ds(\phi,\B)\,.\]
\end{proof}

\begin{theorem}\label{sufficient condition}
Suppose $\phi$ is an equation in the language of commutative BCK-algebras, and consider the sequence $d_n:=\ds(\phi, \C_n)$. Let \[\mathcal{D}_\phi=\{d_m\mid d_m<1\}\,.\] If $\mathcal{D}\phi$ has a maximum, say $d_{m_0}$, then $\phi$ has finite satisfiability gap $\epsilon=1-d_{m_0}$ among commutative BCK-algebras.
\end{theorem}

\begin{proof}
Let $\A$ be a finite commutative BCK-algebra. By Theorem \ref{RT80} we have \[\A\cong \C_{j_1}\times\C_{j_2}\times\cdots\times \C_{j_k}\,,\] where each $j_i\geq 2$. Applying Lemma \ref{ds is mult}, we have \[\ds(\phi,\A)=\prod_{i=1}^k d_{j_i}\,.\] If $\{d_{j_i}\}_{i=1}^k\cap \mathcal{D}_\phi=\emptyset$, then $d_{j_i}=1$ for $1\leq i\leq k$, and we have $\ds(\phi, \A)=1$.

On the other hand, if $\{d_{j_i}\}_{i=1}^k\cap \mathcal{D}_\phi\neq\emptyset$, then for each $1\leq i\leq k$ we have either $d_{j_i}=1$ or $d_{j_i}\leq d_{m_0}$. So $\ds(\phi,\A)\leq d_{m_0}=1-\epsilon$, and $\phi$ has finite satisfiability gap $\epsilon$.
\end{proof}

\section{Equations in one variable}\label{sec:eqns 1 var}

\subsection{Excluded middle degree}\label{sec:emd}

Suppose that $\A$ is a bounded commutative BCK-algebra of order $n$ and put \[E(\A)=\{\, x\in\A\,\mid\, x\join\neg x=1\,\}\,.\]  Define the \emph{excluded middle degree of $\A$}, denoted $\emd(\A)$, to be the degree of satisfiability of the equation $x\join\neg x=1$; that is, 
\[\emd(\A)=\frac{|E(\A)|}{n}\,.\] This is just the probability that a randomly chosen element satisfies the law of the excluded middle.

 It is easy to check that $0\join\neg 0 =1=1\join\neg 1$, so $|E(\A)|\geq 2$. It is known that $E(\A)=\A$ if and only if $\A$ is positive implicative \cite{it78}, meaning $\emd(\A)=1$ if and only if $\A$ is positive implicative. We show that if $\A$ is not positive implicative, then $\emd(\A)\leq \frac{2}{3}$. That is, the law of the excluded middle has finite satisfiability gap $\frac{1}{3}$ in the language of bounded commutatutive BCK-algebras.

\begin{lemma}\label{emd C_n}
$\emd(\C_n)=\frac{2}{n}$
\end{lemma}

\begin{proof}
As observed above, we have $\{0, n-1\}\in E(\C_n)$. Take $k\in C_n\setminus \{0,n-1\}$. Then 
\[k\join \neg k=k\join (n-1-k)=\max\{k, n-k-1\}\neq n-1\] so $k\notin E(\C_n)$. Thus, $|E(\C_n)|=2$ and $\emd(\C_n)=\frac{2}{n}$.
\end{proof}

\begin{theorem}\label{EM has finite gap}
The equation $x\join\neg x=1$ has finite satisfiability gap $\frac{1}{3}$ in the language of bounded commutative BCK-algebras.
\end{theorem}

\begin{proof}
Recall the set $\mathcal{D}$ defined in Theorem \ref{sufficient condition}. From Lemma \ref{emd C_n} we have $d_n=\frac{2}{n}$. So $d_2=1$ and $\mathcal{D}$ has maximum $d_3=\frac{2}{3}$. Apply Theorem \ref{sufficient condition}.

The gap is realized by the algebra $\C_3$.
\end{proof}

\begin{prop}\label{x=1}
In the language of bounded commutative BCK-algebras, the equations $x=1$ and $\neg x=1$ both have finite satisfiability gap $\frac{1}{2}$.
\end{prop}

\begin{proof}
We could prove this with an application of Theorem \ref{sufficient condition}, but we give here a simple argument. Every bounded commutative BCK-algebra $\A$ has at least two elements, so 
\[\ds(x=1, \A)=\frac{|\{y\in\A\,\mid\, y=1\}|}{|\A|}=\frac{1}{|\A|}\leq \frac{1}{2}\,.\]

Similarly, since double negation holds in bounded commutative BCK-algebras \cite{it78}, the equation $\neg x=1$ holds if and only if $x=0$, and we have $\ds(\neg x=1,\A)\leq\frac{1}{2}$.

Both gaps are realized by the algebra $\2$.
\end{proof}

\begin{rem}
We include Proposition \ref{x=1} in order to compare our situation with that of Heyting algebras discussed in \cite{BK22}. In that paper, the authors prove that, up to logical equivalence, the only formulas in one variable that have finite satisfiability gap in the language of Heyting algebras are $x=1$, $\neg x=1$, and $x\join\neg x=1$. To do this they use the characterization of the free Heyting algebra on one generator as the Rieger-Nishimura lattice. 

It is known that bounded commutative BCK-algebras are term-equivalent to MV-algebras \cite{mundici86}, and the free $n$-generated MV-algebra $\mc{MV}_n$ is (isomorphic to) the MV-algebra of McNaughton functions $[0,1]^n\to [0,1]$ \cite{CDM00}. However, even with this knowledge, it is not clear to the author how to characterize the formulas in one variable with finite satisfiability gap in the language of bounded commutative BCK-algebras. 

The situation is more dire in the language of BCK-algebras -- or even commutative BCK-algebras -- as the author is unaware of any characterization of the free algebras in these cases.
\end{rem}

\subsection{Double negation degree}\label{sec:dnd}

Let $\A$ be a bounded BCK-algebra of order $n$ and put \[D(\A)=\{\, x\in\A\,\mid\, \neg\neg x=x\,\}\,.\]  Define the \emph{double negation degree of $\A$}, denoted $\dnd(\A)$, to be the degree of satisfiability of the equation $\neg\neg x=x$; that is, 
\[\dnd(\A)=\ds(\text{DN},\A)=\frac{|D(\A)|}{n}\,.\] This is just the probability that a randomly chosen element is fixed under double negation.

 It is easy to check that $\neg\neg 0 =0$ and $\neg\neg 1=1$, so $|D(\A)|\geq 2$. If $\A$ is commutative then $D(\A)=\A$ \cite{it78}, and so $\dnd(\A)=1$. But if $\A$ is non-commutative then we have \[\frac{2}{n}\leq \dnd(\A)\leq \frac{n-1}{n}\,.\] We will show these bounds are sharp for $n\geq 3$. For example, one computes $\dnd(\PI)=\frac{2}{3}$.

\begin{theorem}\label{min dnd}
For each $n\geq 3$, there is a bounded non-commutative BCK-algebra $\B$ of order $n$ realizing the lower bound $\dnd(\B)=\frac{2}{n}$.
\end{theorem}

\begin{proof}
Let $\A$ be any BCK-algebra of order $n-1$ and consider Is\'{e}ki's extension $\A\oplus\top$. Notice that \[\neg\neg x=\top\bdot(\top\bdot x)=\begin{cases}0&\text{ if $x\in \A$}\\\top&\text{ if $x=\top$}\end{cases}\,,\] so only 0 and $\top$ are fixed by double negation. Hence, $\dnd(\A\oplus\top)=\frac{2}{n}$.
\end{proof}

To show that the upper bound for the double negation degree is also obtained, we need a new family of algebras. 
 
We define a one-element extension $\D_n$ of $\C_n$, with new top element $n$, as follows: 
\begin{align*}
n\bdot 0&=n\\
n\bdot k&=n-k-1 \tag{$k\in\{1,2,\ldots n-2\}$}\\
n\bdot (n-1)&=1\\
k\bdot n &=0 \tag{$k\in \C_n$}
\end{align*} For small orders, these $\D_n$'s are known to be BCK-algebras. For example, $\D_3$ is labelled as $B_{4-3-1}$ in \cite{mj94}, while $\D_4$ is labelled as $B_{5-4-9}$. Their Cayley tables are shown in Table \ref{tab:D}.
\begin{table}[h]
{\centering
\begin{tabular}{c||cccc}
$\bdot$  & 0              & 1 & 2 & 3\\\hline\hline
0             & 0              & 0 & 0 & 0\\
1             & 1              & 0 & 0 & 0\\
2             & 2              & 1 & 0 & 0 \\
3             & 3              & 1 & 1 & 0
\end{tabular}
\hspace{1cm}
\begin{tabular}{c||ccccc}
$\bdot$  & 0              & 1 & 2 & 3 & 4\\\hline\hline
0             & 0              & 0 & 0 & 0 & 0\\
1             & 1              & 0 & 0 & 0 & 0\\
2             & 2              & 1 & 0 & 0 & 0\\
3             & 3              & 2 & 1 & 0 & 0\\
4             & 4              & 2 & 1 & 1 & 0
\end{tabular}
\caption{\label{tab:D}the algebras $\D_3$ and $\D_4$}
}
\end{table}

The proof that $\D_n$ is a BCK-algebra in general is tedious and given in an appendix. The next theorem shows that this family of algebras achieves the upper bound for double negation degree.

\begin{theorem}\label{max dnd}
For each $n\geq 3$, there is a bounded non-commutative BCK-algebra $\A$ of order $n$ realizing the upper bound $\dnd(\A)=\frac{n-1}{n}$.
\end{theorem}

\begin{proof}
Consider the algebra $\D_n$ which is shown to be a non-commutative BCK-algebra in Lemma \ref{D_n is BCK}. For $k\in\{1, 2, \ldots, n-2\}$, we have $n\bdot k = n-k-1\in \{1, 2, \ldots, n-2\}$, so \[n\bdot(n\bdot k)=n\bdot (n-k-1)=n-(n-k-1)-1=k\,.\] Thus, $\{0, 1, 2, \ldots, n-2, n\}\subseteq D(\D_n)$.

However, $n\bdot\bigl(n\bdot (n-1)\bigr)=n\bdot 1=n-2\neq n-1$, so $n-1\notin D(\D_n)$. Thus, $\dnd(\D_n)=\frac{n}{n+1}$.

\end{proof}

\begin{cor}
The equation $\neg\neg x=x$ has no finite satisfiability gap in the language of bounded BCK-algebras.
\end{cor}

By Theorems \ref{min dnd} and \ref{max dnd}, we can make the double negation degree as close to 0 as we like or as close to 1 as we like. For an algebra of order $n$, the possible double negation degrees are \[\Bigl\{\,\frac{2}{n}\,,\, \frac{3}{n}\,,\,\ldots\,,\, \frac{n-1}{n}\,,\, 1\,\Bigr\}\,.\] Empirically, the author has observed that for orders 3, 4, and 5, every possible double negation degree is obtained by some algebra. We conjecture this is true in general but do not yet have a proof.

\begin{con}
For each $n\geq 3$, every possible double negation degree is obtained by some bounded BCK-algebra. Consequently, every rational in $\bb{Q}\cap (0,1]$ is the double negation degree of some BCK-algebra.
\end{con}

\section{Equations in two variables}\label{sec:eqns 2 var}
 
\subsection{Commuting degree}\label{sec:cd}

Let $\A$ be a BCK-algebra of order $n$. Let \[C(\A)=\{\, (x,y)\in A^2\,\mid\, x\meet y= y\meet x\,\}\,.\] Define the \emph{commuting degree} of $\A$, denoted $\cd(\A)$, to be the degree of satisfiability of the commutativity equation $x\meet y=y\meet x$. That is, \[\cd(\A)=\frac{|C(\A)|}{n^2}\,.\] This is the probability that two elements chosen uniformly randomly (with replacement) commute with one another.

In any BCK-algebra $\A$, we have $x\meet 0=0\meet x$ for all $x\in \A$, and every element commutes with itself by (BCK3). So every pair of the form $(x,0)$, $(0,x)$, and $(x,x)$ is in $C(\A)$. Thus, $|C(\A)| \geq 3n-2$ and \[\cd(\A)\geq \frac{3n-2}{n^2}\,.\] Of course, if $\A$ is commutative then $\cd(\A)=1$, but if $\A$ is non-commutative we have
\[\frac{3n-2}{n^2}\leq \cd(\A)\leq \frac{n^2-2}{n^2}\,.\] We will show these bounds are sharp.

Suppose $\U=\bigsqcup_{\lambda\in\Lambda}\A_\lambda$ is a BCK-union of a family of BCK-algebras. If $a,b\in \U$ where $a\in \A_\lambda$ and $b\in \A_\mu$ with $\lambda\neq\mu$, then $a$ and $b$ necessarily commute:
\begin{align*}
a\meet b &= b\bdot (b\bdot a)=b\bdot b=0\\
b\meet a &= a\bdot (a\bdot b)=a\bdot a=0\,.
\end{align*} Thus, heuristically, we expect that the commuting degree is larger when an algebra has many incomparable elements, and smaller when an algebra has many comparable elements. Indeed, we will see this is the case, but we note that if all non-zero elements are incomparable, then the underlying poset admits a unique BCK-product which is necessarily commutative \cite{it76}, so the commuting degree is 1 for such an algebra.

\begin{prop}\label{cd of union}
Suppose $\A$ and $\B$ are BCK-algebras with $|\A|=n$, $|\B|=m$, $\cd(\A)=\frac{k}{n^2}$, and $\cd(\B)=\frac{\ell}{m^2}$. Then 
\[\cd(\A\sqcup\B)=\frac{k+\ell+2(n-1)(m-1)-1}{(n+m-1)^2}\,.\]
\end{prop}

\begin{proof}
First note that $|\A\sqcup\B|=n+m-1$ since the BCK-union shares a 0 element, but the algebras are otherwise disjoint. 

Next, among the $k$ commuting pairs in $\A$ and the $\ell$ commuting pairs in $\B$, the pair $(0,0)$ is double counted, so $C(\A\sqcup\B)\geq k+\ell-1$. But by our observations above, all elements of $\A$ commute with all elements of $\B$. Thus, for all $x\in \A\setminus\{0\}$ and all $y\in\B\setminus\{0\}$, the pairs $(x,y)$ and $(y,x)$ are in $C(\A\sqcup\B)$. There are $2(n-1)(m-1)$ such pairs.

Hence, \[\cd(\A\sqcup\B)=\frac{k+\ell+2(n-1)(m-1)-1}{(n+m-1)^2}\,.\]
\end{proof}

\begin{cor}\label{union with 2}
Let $\A$ be a BCK-algebra with $|\A|=n$ and $\cd(\A)=\frac{k}{n^2}$. Then 
\[\cd(\A\sqcup\2)=\frac{k+2n+1}{(n+1)^2}\,.\]
\end{cor}

\begin{proof}
The algebra $\2$ is commutative, so apply Proposition \ref{cd of union} with $\ell=4$ and $m=2$.
\end{proof}

\begin{theorem}\label{max cd}
For each $n\geq 3$, there is a non-commutative BCK-algebra of order $n$ realizing the maximum commuting degree $\frac{n^2-2}{n^2}$.
\end{theorem}

\begin{proof} For $n=3$, recall the algebra $\PI$ defined in Table \ref{tab:algs}. The only non-commuting pairs are $(1,2)$ and $(2,1)$, so $\PI$ has commuting degree $\frac{7}{9}$.

Define a family of algebras $\B_n$ by 
\begin{align*}
\B_3&=\PI\\ 
\B_n&=\B_{n-1}\sqcup\2 \text{ for $n>3$\,.}
\end{align*}

For $n>3$, assume that $\cd(\B_{n-1})=\frac{(n-1)^2-2}{(n-1)^2}=\frac{n^2-2n-1}{(n-1)^2}$. By Corollary \ref{union with 2} we have
\[\cd(\B_n)=\cd(\B_{n-1}\sqcup\2)=\frac{(n^2-2n-1)+\bigl(2(n-1)+1\bigr)}{n^2}=\frac{n^2-2}{n^2}\,,\] and the result follows by induction.
\end{proof}

If we label the atoms of $\B_n$ by $\{a_i\}_{i=1}^{n-2}$, then Figure \ref{fig:B_n} shows the Hasse diagram, verifying our intuition that commuting degree is large when the algebra has many incomparable elements.
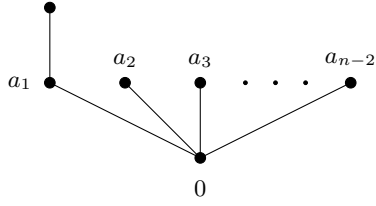
\begin{figure}[h]
\centering
\begin{tikzpicture}
\filldraw (0,0) circle (2pt);
\filldraw (-2,1) circle (2pt);
\filldraw (-2,2) circle (2pt);
\filldraw (-1,1) circle (2pt);
\filldraw (0,1) circle (2pt);
\filldraw (2,1) circle (2pt);
\draw [-] (0,0) -- (-2,1) -- (-2,2);
\draw [-] (0,0) -- (-1,1);
\draw [-] (0,0) -- (0,1);
\draw [-] (0,0) -- (2,1);
	\node at (0,-.4) {\small 0};
	\node at (-2.4, 1) {\small $a_1$};
	\node at (-1, 1.3) {\small $a_2$};
	\node at (0, 1.3) {\small $a_3$};
\filldraw (.6,1) circle (.7pt);
\filldraw (1,1) circle (.7pt);
\filldraw (1.4,1) circle (.7pt);
	\node at (2, 1.3) {\small $a_{n-2}$};
\end{tikzpicture}
\caption{Hasse diagram for $\B_n$}
\label{fig:B_n}
\end{figure}

\begin{cor}
In the language of BCK-algebras, the equation $x\meet y = y\meet x$ has no finite satisfiability gap.
\end{cor}

Next we consider the lower bound for commuting degree.

\begin{prop}\label{cd of sum}
Suppose $\A$ is a BCK-algebra with $|\A|=n$ and $\cd(\A)=\frac{k}{n^2}$. Then 
\[\cd(\A\oplus\top)=\frac{k+3}{(n+1)^2}\,.\]
\end{prop}

\begin{proof}
We must have $(0,\top)$, $(\top, 0)$, and $(\top,\top)$ in $C(\A\oplus\1)$, but note that $\top$ does not commute with any non-zero element of $\A$:
\begin{align*}
x\meet\top &=\top\bdot(\top\bdot x)=\top\bdot\top=0\,,\\
\top\meet x&= x\bdot(x\bdot\top)=x\bdot 0=x
\end{align*} for any non-zero $x\in\A$. Thus, $|C(\A\oplus\top)|=k+3$ and $\cd(\A\oplus\top)=\frac{k+3}{(n+1)^2}$.
\end{proof}

\begin{theorem}\label{min cd}
For each $n\geq 3$, there is a BCK-algebra of order $n$ realizing the minimum commuting degree $\frac{3n-2}{n^2}$.
\end{theorem}

\begin{proof} For $n=3$, the algebra $\PI$ has commuting degree $\frac{7}{9}$, as mentioned in the proof of Theorem \ref{max cd}.

Define a new family of BCK-algebras as follows:
\begin{align*}
\M_3 &= \PI \text{ and}\\ 
\M_n&=\M_{n-1}\oplus\top\text{ for $n>3$\,.}
\end{align*} 

Assume $\cd(\M_{n-1})=\frac{3(n-1)-2}{(n-1)^2}$ for some $n>3$. By Proposition \ref{cd of sum}, we have
\[\cd(\M_n)=\cd(\M_{n-1}\oplus\top)=\frac{\bigl(3(n-1)-2\bigr)+3}{n^2}=\frac{3n-2}{n^2}\,,\] and the result follows by induction.
\end{proof}

Since the $\M_n$'s are linear, the above result verifies our intuition that commuting degree is small when there are many comparable elements. 

To borrow terminology from \cite{BK22}, one might say that the $\B_n$'s are ``deceptively'' non-commutative: they are as close as possible to being commutative without actually being commutative. On the other hand, the $\M_n$'s are transparently non-commutative, being as far away as possible from commutativity.

For a BCK-algebra $\A$ of order $n$, the possible commuting degrees are \[\Bigl\{\,\frac{3n-2}{n^2}\,,\, \frac{3n}{n^2}\,,\, \frac{3n+2}{n^2}\,,\, \ldots\,,\, \frac{n^2-2}{n^2}\,,\, 1\,\Bigr\}\,.\] Empirically, the author has observed that for orders 3, 4, and 5, every possible commuting degree is obtained by some algebra. We conjecture this is true in general but do not yet have a proof.

\begin{con}
For each $n\geq 3$, every possible commuting degree is obtained by some BCK-algebra.
\end{con}

\subsection{Positive implicative degree}\label{sec:pid}

Consider a BCK-algebra $\A$ of order $n$ and recall the equation (E$_1$). Let \[P(\A)=\{\, (x,y)\in\A^2\,\mid\, (x\bdot y)\bdot y=x\bdot y\,\}\,,\] and define the \emph{positive implicative degree of $\A$}, denoted $\pid(\A)$, to be the degree of satisfiability of the equation $(x\bdot y)\bdot y=x\bdot y$; that is, 
\[\pid(\A)=\frac{|P(\A)|}{n^2}\,.\] One checks that all pairs of the form $(0,x)$, $(x,0)$, and $(x,x)$ are in $P(\A)$ for all $x\in \A$, so $|P(\A)|\geq 3n-2$. But also, any pair $(x,y)$ such that $x\bdot y=0$ is in $P(\A)$. That is, if $x\leq y$, then $(x,y)\in P(\A)$. This makes $P(\A)$ generally harder to count than its counterpart $C(\A)$.

The simplest possible case is that all non-zero elements of $\A$ are incomparable. However, such a poset admits a unique BCK-product and the corresponding algebra will always be positive implicative (in fact, implicative)\cite{it76}, so $\pid(\A)=1$ for such an algebra.

Thus, in order for $\A$ to fail (E$_1$), there must be at least one pair $(x,y)$ of non-zero elements that are comparable. We will show that, among algebras that fail (E$_1$), the maximum positive implicative degree is achieved by an algebra having exactly one pair of non-zero elements $(x,y)$ that are comparable, though we note this is not the only method to achieve the maximum.

\begin{prop}\label{pid of union}
Suppose $\A$ and $\B$ are BCK-algebras with $|\A|=n$, $|\B|=m$, $\pid(\A)=\frac{k}{n^2}$, and $\pid(\B)=\frac{\ell}{m^2}$. Then 
\[\pid(\A\sqcup\B)=\frac{k+\ell+2(n-1)(m-1)-1}{(n+m-1)^2}\,.\] Consequently, 
\[\pid(\A\sqcup\2)=\frac{k+2n+1}{(n+1)^2}\,.\]
\end{prop}

\begin{proof} The proof is similar to that of Proposition \ref{cd of union}. By the definition of $\bdot$ on $\A\sqcup\B$, if $x\in\A\setminus\{0\}$ and $y\in \B\setminus\{0\}$, we have
\begin{align*}
(x\bdot y)\bdot y&=x\bdot y\\
(y\bdot x)\bdot x&=y\bdot x
\end{align*} and so $(x,y)$ and $(y,x)$ are in $P(\A\sqcup\B)$ for all such pairs.

Thus, \[\pid(\A\sqcup\B)=\frac{k+\ell+2(n-1)(m-1)-1}{(n+m-1)^2}\,.\]

Finally, we note $\2$ is positive implicative, so $\pid(\2)=1$.
\end{proof}

For $n\geq 3$, define a family of algebras as follows:
\begin{align*}
\P_3&=\TC\\ 
\P_n&=\P_{n-1}\sqcup\2\text{ for $n>3$\,.}
\end{align*} If we let $\{a_i\}_{i=1}^{n-2}$ denote the atoms of $\P_n$, the Hasse diagram for $\P_n$ was already shown in Figure \ref{fig:B_n}. Notice there is exactly one pair $(x,y)$ of non-zero elements with $x<y$.

\begin{theorem}\label{max pid}
For each $n\geq 3$, there is a BCK-algebra $\A$ of order $n$ realizing the upper bound $\pid(\A)=\frac{n^2-1}{n^2}$. Thus, in the language of BCK-algebras, the equation $(x\bdot y)\bdot y=x\bdot y$ has no finite satisfiability gap.
\end{theorem}

\begin{proof} For $n=3$, one computes $\pid(\TC)=\frac{8}{9}$.

For $n>3$, assume $\pid(\P_{n-1})=\frac{(n-1)^2-1}{(n-1)^2}=\frac{n^2-2n}{(n-1)^2}$. By Prop \ref{pid of union}, we have
\[\pid(\P_n)=\pid(\P_{n-1}\sqcup\2)=\frac{(n^2-2n)+\bigl(2(n-1)+1\bigr)}{n^2}=\frac{n^2-1}{n^2}\,,\] and the result follows by induction.
\end{proof}

As mentioned, this maximum value can be achieved in other ways. We outline one other way here.

\begin{prop}
Let $\A$ be a BCK-algebra with $|\A|=n$ and $\pid(\A)=\frac{k}{n^2}$. Then 
\[\pid(\A\oplus\top)=\frac{k+2n+1}{(n+1)^2}\,.\]
\end{prop}

\begin{proof} By the definition of $\bdot$ on $\A\oplus\top$, we have
\begin{align*}
(\top\bdot x)\bdot x&=\top\bdot x\\
(x\bdot\top)\bdot \top&=0\bdot\top=0=x\bdot\top
\end{align*} for all $x\in\A$. So $(x,\top), (\top,x)\in P(\A\oplus\top)$ for all $x\in \A$, and certainly $(\top,\top)\in P(\A\oplus\top)$. 

Thus $\pid(\A\oplus\top)=\frac{k+2n+1}{(n+1)^2}$.
\end{proof}

If we now define a new family of linear BCK-algebras by
\begin{align*}
\P_3'&=\TC\\ 
\P_n'&=\P_{n-1}'\oplus\top\text{ for $n>3$\,,}
\end{align*} and mimic the proof of Theorem \ref{max pid}, we see the family $\P_n'$ also achieve the maximum positive implicative degree for each $n\geq 3$.

Next we consider the minimum positive implicative degree. Suppose $\A$ is a BCK-algebra of order $n$ with a unique atom $a$. Label the elements of $\A$ by $\{0, a, b_1, b_2, \ldots, b_{n-2}\}$. So $0<a<b_i$ for $i=1,2,\ldots, n-2$. In Table \ref{tab:pairs in P(A)} we indicate the pairs in $\A\times\A$ that must be in $P(\A)$.
\begin{table}[h]
{\centering
\begin{tabular}{c||ccccccc}
                 & 0                      & $a$       & $b_1$           & $b_2$           & $\cdots$ & $b_{n-3}$                  & $b_{n-2}$ \\\hline\hline
0               & $(0,0)$            & $(0,a)$ & $(0,b_1)$     & $(0,b_2)$     & $\cdots$ & $(0,b_{n-3})$           & $(0,b_{n-2})$       \\
$a$           & $(a,0)$            & $(a,a)$ & $(a,b_1)$     & $(a,b_2)$     & $\cdots$ & $(a,b_{n-3})$           & $(a,b_{n-2})$ \\
$b_1$       & $(b_1,0)$        &              & $(b_1,b_1)$ &                      &                &                                   &  \\
$b_2$       & $(b_2,0)$        &              &                      & $(b_2,b_2)$ &                &                                   & \\
$\vdots$   & $\vdots$          &              &                      &                     & $\ddots$ &                                   & \\
$b_{n-3}$ & $(b_{n-3},0)$  &             &                      &                     &                 & $(b_{n-3},b_{n-3})$ & \\
$b_{n-2}$ & $(b_{n-2},0)$  &             &                      &                     &                 &                                   & $(b_{n-2},b_{n-2})$
\end{tabular}
\caption{\label{tab:pairs in P(A)} pairs which must be in $P(\A)$}}
\end{table} If we let $T_k$ denote the $k^{\text{th}}$ triangular number, then there are $T_{n-2}$ empty spaces below the diagonal of the table, and $T_{n-3}$ empty spaces above the diagonal. Thus, we have the lower bound
\[\pid(\A)\geq\frac{n^2-T_{n-2}-T_{n-3}}{n^2}=\frac{4n-4}{n^2}\,.\] We will see this lower bound is sharp.

Now, if $\A$ has two atoms, say $a_1$ and $a_2$, let $\{0, a_1, a_2, b_1, b_2, \ldots, b_{n-3}\}$ be the carrier set of $\A$. Then for each $i\in\{1, 2, \ldots, n-3\}$, either $a_1<b_i$ or $a_2<b_i$, meaning either $(a_1, b_i)\in P(\A)$ or $(a_2, b_i)\in P(\A)$. Also, in any BCK-algebra we always have $x\bdot y\leq x$ for all $x$ and $y$ \cite{it78}, which forces $a_1\bdot a_2=a_1$ and $a_2\bdot a_1=a_2$. Thus, $(a_1,a_2), (a_2, a_1)\in P(\A)$, and we have 
\[|P(\A)|\geq (3n-2)+(n-3)+2=4n-3>4n-4\,.\] So the minimum value of $\pid$ can only be achieved by an algebra with a unique atom.

For $n\geq 3$, consider the $n$-element poset $Q_n$ in Figure \ref{fig:Q_n}.
\begin{figure}[h]
\centering
\begin{tikzpicture}
\filldraw (0,0) circle (2pt);
\filldraw (0,.8) circle (2pt);
\filldraw (-2,2) circle (2pt);
\filldraw (-1,2) circle (2pt);
\filldraw (0,2) circle (2pt);
\filldraw (2,2) circle (2pt);
\draw [-] (0,0) -- (0,1);
\draw [-] (0,.8) -- (-2,2);
\draw [-] (0,.8) -- (-1,2);
\draw [-] (0,.8) -- (0,2);
\draw [-] (0,.8) -- (2,2);
	\node at (0,-.4) {\small 0};
	\node at (.3,.7) {\small $a$};
	\node at (-2, 2.3) {\small $b_1$};
	\node at (-1, 2.3) {\small $b_2$};
	\node at (0, 2.3) {\small $b_3$};
\filldraw (.6,2) circle (.7pt);
\filldraw (1,2) circle (.7pt);
\filldraw (1.4,2) circle (.7pt);
	\node at (2, 2.3) {\small $b_{n-2}$};
\end{tikzpicture}
\caption{the poset $Q_n$}
\label{fig:Q_n}
\end{figure}
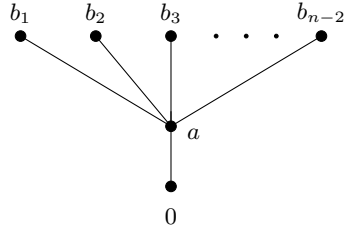 Define an operation on $Q_n$ by 
\[x\bdot y=\begin{cases} 0 &\text{if $x\leq y$}\\ x &\text{if $y=0$}\\ a &\text{if $x=b_i, y=a$}\\ a &\text{if $x=b_i, y=b_j$ with $i\neq j$}\end{cases}\,.\] This is a (commutative) BCK-algebra for each $n\geq 3$ \cite{it76}, and we will denote it by $\Q_n$. Notice that $\Q_3\cong \TC$, and that $\Q_n$ is a subalgebra of $\Q_{n+1}$.

\begin{theorem}\label{min pid atom}
For $n\geq 3$, there is a BCK-algebra $\A$ of order $n$ realizing the lower bound $\pid(\A)=\frac{4n-4}{n^2}$.
\end{theorem}

\begin{proof} We saw in Theorem \ref{max pid} that $\pid(\TC)=\frac{8}{9}=\frac{4(3)-4}{9}$.

For $n>3$, assume $\pid(\Q_{n-1})=\frac{4(n-1)-4}{(n-1)^2}=\frac{4n-8}{(n-1)^2}$. Label the elements of $\Q_n$ as $\{0, a, b_1, b_2, \ldots, b_{n-3}, b_{n-2}\}$ and view $\Q_{n-1}$ as a subalgebra of $\Q_n$ with carrier set $\{0, a, b_1, b_2, \ldots, b_{n-3}\}$. 

Then $P(\Q_{n-1})\subseteq P(\Q_n)$. The remaining pairs we need to consider must have the element $b_{n-2}$ in one of the coordinates. 

By definition of $\bdot$, we have
\begin{align*}
(b_{n-2}\bdot a)\bdot a&=a\bdot a=0\neq a=b_{n-2}\bdot a\,,\\
(b_{n-2}\bdot b_i)\bdot b_i&=a\bdot b_i=0\neq a=b_{n-2}\bdot b_i\,,\tag{$i\neq n-2$}\\
(b_i\bdot b_{n-2})\bdot b_{n-2}&=a\bdot b_{n-2}=0\neq a=b_i\bdot b_{n-2}\,,\tag{$i\neq n-2$}
\end{align*} meaning $(b_{n-2},a)\notin P(\Q_n)$ and $(b_{n-2},b_i), (b_i,b_{n-2})\notin P(\Q_n)$ for $i\neq n-2$. 

On the other hand, we know $(b_{n-2}, 0)$, $(0, b_{n-2})$, and $(b_{n-2}, b_{n-2})$ are in $P(\Q_n)$. And since $a\leq b_{n-2}$ we have $(a,b_{n-2})\in P(\Q_n)$ as well.

Thus, $\pid(\Q_n)=\frac{(4n-8)+4}{n^2}=\frac{4n-4}{n^2}$ and the result follows by induction.
\end{proof}

As a special case, suppose that $\A$ is a linear BCK-algebra of order $n$. Label the elements of $\A$ by $\{0, a_1, a_2, \ldots, a_{n-1}\}$ with $0<a_1<a_2<\cdots < a_{n-1}$. In Table \ref{tab:pairs in P(A)2} we indicate the pairs in $\A\times\A$ that must be in $P(\A)$.
\begin{table}[h]
{\centering
\begin{tabular}{c||cccccc}
                 & 0                      & $a_1$            & $a_2$           & $a_3$            & $\cdots$ & $a_{n-1}$                \\\hline\hline
0               & $(0,0)$            & $(0,a_1)$     & $(0,a_2)$      & $(0,a_3)$      & $\cdots$ & $(0,a_{n-1})$           \\
$a_1$       & $(a_1,0)$        & $(a_1,a_1)$ & $(a_1,a_2)$  & $(a_1,a_3)$  & $\cdots$ & $(a_1,a_{n-1})$           \\
$a_2$       & $(a_2,0)$        &                      & $(a_2,a_2)$  & $(a_2,a_3)$  & $\cdots$ & $(a_2,a_{n-1})$          \\
$a_3$       & $(a_3,0)$        &                      &                       & $(a_3,a_3)$  & $\cdots$ & $(a_3,a_{n-1})$                                  \\
$\vdots$   & $\vdots$          &                     &                       &                      & $\ddots$ & $\vdots$                                  \\
$a_{n-1}$ & $(a_{n-1},0)$  &                    &                       &                      &                 & $(a_{n-1},a_{n-1})$ \\
\end{tabular}
\caption{\label{tab:pairs in P(A)2} pairs which must be in $P(\A)$}}
\end{table} As we saw before, there are $T_{n-2}$ empty spaces below the diagonal of the table. Thus, we have the lower bound
\[\pid(\A)\geq\frac{n^2-T_{n-2}}{n^2}=\frac{n^2+3n-2}{2n^2}\,,\] and we will show that this lower bound is sharp.

\begin{theorem}\label{min pid linear}
Among linear BCK-algebras, there is an algebra $\A$ of order $n$ realizing the lower bound $\pid(\A)=\frac{n^2+3n-2}{2n^2}$ for $n\geq 3$.
\end{theorem}

\begin{proof} Recall the family of algebras $\C_n$ defined in Table \ref{tab:C_n}, and note that $\C_3=\TC$. We saw in Theorem \ref{max pid} that $\pid(\TC)=\frac{8}{9}=\frac{3^2+3(3)-2}{2(3)^2}$.

For $n>3$, assume $\pid(\C_{n})=\frac{n^2+3n-2}{2n^2}$. Label the elements of $\C_{n+1}$ as $\{0, 1,2,\ldots, n-1,n\}$. We can view $\C_{n}$ as a subalgebra of $\C_{n+1}$ with carrier set $\{0, 1,2,\ldots, n-1\}$. 

Then $P(\C_n)\subseteq P(\C_{n+1})$. The remaining pairs we need to consider must have the element $n$ in one of the coordinates. For any $x\in C_n$, we know that $(x,n)\in P(\C_n)$ since $x\leq n$, adding $(n+1)$-many pairs to $P(\C_{n+1})$. We also know that $(n,0)\in P(\C_{n+1})$.

On the other hand, for $x\in\{1,2,\ldots, n-1\}$ we have \[(n\bdot x)\bdot x=\left.\begin{cases} n-2x&\text{if $n>2x$}\\0&\text{if $n\leq2x$}\end{cases}\right\}\neq n-x=n\bdot x\,,\] meaning $(n,x)\notin P(\C_{n+1})$.

Thus, \begin{align*}|P(\C_{n+1})|=|P(\C_n)|+(n+1)+1&=\frac{n^2+3n-2}{2}+n+2\\
&=\frac{n^2+5n+2}{2}\\
&=\frac{(n+1)^2+3(n+1)-2}{2}\,,
\end{align*} and therefore we have $\pid(\C_{n+1})=\frac{(n+1)^2+3(n+1)-2}{2(n+1)^2}$. The result holds by induction.
\end{proof}

\begin{cor}
If $\A$ is a finite linear BCK-algebra, then $\frac{1}{2} < \pid(\A)\leq 1$.
\end{cor}

We close this section by observing that while the positive implicative equation has no finite satisfiability gap in the language of BCK-algebras, it does have a gap in the language of commutative BCK-algebras. 

\begin{cor}\label{E_1 has finite gap in cBCK}
In the language of commutative BCK-algebras, the equation (E$_1$) has finite satisfiability gap $\frac{1}{9}$.
\end{cor}

\begin{proof}
We saw in Theorem \ref{min pid linear} that $\pid(\C_n)$ is a decreasing sequence, so $\mathcal{D}$ has maximum $\frac{8}{9}$. Applying Theorem \ref{sufficient condition} gives the result.
\end{proof}

\subsection{Implicative degree}\label{sec:id}

Consider a BCK-algebra $\A$ of order $n$ and recall the equation (I). Let \[I(\A)=\{\, (x,y)\in\A^2\,\mid\, x\bdot (y\bdot x)=x\,\}\,,\] and define the \emph{implicative degree of $\A$}, denoted $\id(\A)$, to be the degree of satisfiability of the equation $x\bdot (y\bdot x)=x$; that is, 
\[\id(\A)=\frac{|I(\A)|}{n^2}\,.\] One checks that all pairs of the form $(0,x)$, $(x,0)$, and $(x,x)$ are in $I(\A)$ for all $x\in \A$, so $|I(\A)|\geq 3n-2$. But also, any pair $(x,y)$ such that $y\bdot x=0$ is in $I(\A)$. That is, if $x\geq y$, then $(x,y)\in I(\A)$. Counting the elements of $I(\A)$ is therefore similar to counting the elements of $P(\A)$.

As we noted in Section \ref{sec:pid}, if all non-zero elements of $\A$ are incomparable, $\A$ admits a unique BCK-product which is implicative, so $\id(\A)=1$ for such an algebra.

Thus, in order for $\A$ to fail (I), there must be at least one pair $(x,y)$ of non-zero elements that are comparable. We will see that, among algebras that fail (I), the bounds on implicative degree are actually the same as those for positive implicative degree, and the same algebras witness those bounds.

\begin{prop}\label{id of union}
Suppose $\A$ and $\B$ are BCK-algebras with $|\A|=n$, $|\B|=m$, $\id(\A)=\frac{k}{n^2}$, and $\id(\B)=\frac{\ell}{m^2}$. Then 
\[\id(\A\sqcup\B)=\frac{k+\ell+2(n-1)(m-1)-1}{(n+m-1)^2}\,.\] Consequently, 
\[\id(\A\sqcup\2)=\frac{k+2n+1}{(n+1)^2}\,.\]
\end{prop}

\begin{proof} The proof is similar to that of Propositions \ref{cd of union} and \ref{pid of union}. By the definition of $\bdot$ on $\A\sqcup\B$, if $x\in\A\setminus\{0\}$ and $y\in \B\setminus\{0\}$, we have
\begin{align*}
x\bdot (y\bdot x)&=x\bdot y=x\\
y\bdot (x\bdot y)&=y\bdot x=y
\end{align*} and so $(x,y)$ and $(y,x)$ are in $I(\A\sqcup\B)$ for all such pairs.

Thus, \[\id(\A\sqcup\B)=\frac{k+\ell+2(n-1)(m-1)-1}{(n+m-1)^2}\,.\]
\end{proof}

\begin{theorem}\label{max id}
For each $n\geq 3$, there is a BCK-algebra $\A$ of order $n$ realizing the upper bound $\id(\A)=\frac{n^2-1}{n^2}$. Consequently the equation (I) has no finite satisfiability gap in the language of BCK-algebras.
\end{theorem}

\begin{proof} Recall the family of algebras $\P_n$ used in Theorem \ref{max pid}. One computes $\id(\TC)=\frac{8}{9}$, and by an induction we have $\id(\P_n)=\frac{n^2-1}{n^2}$ using Proposition \ref{id of union}.
\end{proof}

We consider now the minimum implicative degree. Suppose $\A$ is a BCK-algebra of order $n$ with a unique atom $a$. The transpose of Table \ref{tab:pairs in P(A)} indicates the pairs in $\A\times\A$ that must be in $I(\A)$, so we obtain the same lower bound of $\id(\A)\geq \frac{4n-4}{n^2}$. The minimum value of $\id$ can only be achieved by an algebra with a unique atom, and the argument is analogous to that used in Section \ref{sec:pid}

\begin{theorem}\label{min id atom}
For $n\geq 3$, there is a BCK-algebra $\A$ of order $n$ realizing the lower bound $\id(\A)=\frac{4n-4}{n^2}$.
\end{theorem}

\begin{proof} Recall the family of algebras $\Q_n$ defined before Theorem \ref{min pid atom}. We have seen that $\id(\Q_3)=\id(\TC)=\frac{8}{9}=\frac{4(3)-4}{9}$.

The remainder of the proof follows the same way as the proof of Theorem \ref{min pid atom}: Assume $\id(\Q_{n-1})=\frac{4n-8}{n^2}$. Note that $I(\Q_{n-1})\subseteq I(\Q_n)$. It is straightforward to check that
\[(a, b_{n-2}), (b_i, b_{n-2}), (b_{n-2}, b_i)\notin I(\Q_n)\] for any $i\in\{1,2,\ldots, n-3\}$, while we must have
\[(0,b_{n-2}), (b_{n-2}, 0), (b_{n-2}, a), (b_{n-2}, b_{n-2})\in I(\Q_n)\,.\] The result follows by induction.
\end{proof}

Consider now the special case where $\A$ is linear. The transpose of Table \ref{tab:pairs in P(A)2} indicates the pairs in $\A\times\A$ that must be in $I(\A)$, and we get the lower bound $\id(\A)\geq\frac{n^2+3n-2}{2n^2}$.

\begin{theorem}\label{min id linear}
Among linear BCK-algebras, there is an algebra $\A$ of order $n$ realizing the lower bound $\pid(\A)=\frac{n^2+3n-2}{2n^2}$ for $n\geq 3$.
\end{theorem}

\begin{proof} Recall the family of algebras $\C_n$ defined in Table \ref{tab:C_n}. We know $\id(\TC)=\frac{8}{9}=\frac{3^2+3(3)-2}{2(3)^2}$.

The remainder of the proof follows the proof of Theorem \ref{min pid linear}: Assume $\pid(\C_{n})=\frac{n^2+3n-2}{2n^2}$, and view $\C_{n}$ as a subalgebra of $\C_{n+1}$ so that $I(\C_n)\subseteq I(\C_{n+1})$. The remaining pairs we need to consider must have the element $n$ in one of the coordinates. The $(n+1)$-many pairs of the form $(n,x)$ must be added to $I(\C_{n+1})$, where $x\in C_{n+1}$. We also know that $(0,n)\in P(\C_{n+1})$.

On the other hand, for $x\in\{1,2,\ldots, n-1\}$ we have \[x\bdot (n\bdot x)=x\bdot (n-x)=\left.\begin{cases} 2x-n&\text{if $2x>n$}\\0&\text{if $2x\leq n$}\end{cases}\right\}\neq x\,,\] meaning $(x,n)\notin I(\C_{n+1})$.

Therefore we have \[\id(\C_{n+1})=\frac{n^2+3n-2}{2n^2}+n+2=\frac{(n+1)^2+3(n+1)-2}{2(n+1)^2}\,.\]
\end{proof}

\begin{cor}
If $\A$ is a finite linear BCK-algebra, then $\frac{1}{2} < \id(\A)\leq 1$.
\end{cor}

\begin{cor}\label{I has finite gap in cBCK}
In the language of commutative BCK-algebras, the equation (I) has finite satisfiability gap $\frac{1}{9}$.
\end{cor}

\section{Appendix}

\begin{lemma}\label{D_n is BCK}
For each $n\geq 3$, the algebra $\D_n$ is a non-commutative BCK-algebra.
\end{lemma}

\begin{proof}
It is easy to check that $\D_n$ satisfies (BCK3), (BCK4), and (BCK5). Since $\D_n$ is a one-element extension of $\C_n$, it only remains to check (BCK1) and (BCK2) with the top element $n$ in place of some (or all) of the variables $x$, $y$ and $z$.

Consider (BCK2). If $x=y=n$, or if $y=n$, it is easy to see (BCK2) is satisfied. So suppose $x=n$. If $y=0$ it is again an easy check, so assume $0<y<n$. We have \[n\bdot y=\begin{cases} n-y-1&\text{ if $y\in\{1,2,\ldots, n-2\}$}\\1&\text{ if $y=n-1$}\end{cases}\,,\] and so 
\[n\bdot(n\bdot y)=\begin{cases}y&\text{ if $y\in\{1,2,\ldots, n-2\}$}\\n-2&\text{ if $y=n-1$}\end{cases}\,.\] Thus, for $y\in\{1,2,\ldots, n-2\}$, we have $\bigl[n\bdot(n\bdot y)\bigr]\bdot y=y\bdot y=0$, and for $y=n-1$ we have $\bigl[n\bdot(n\bdot y)\bigr]\bdot y=(n-2)\bdot (n-1)=0$. So $\D_n$ satisfies (BCK2).

Showing (BCK1) is more involved. Based on how we defined $\bdot$ on $\D_n$, we make the following observations: for all $a,b,c\in \D_n$,
\begin{enumerate}
\item $a\bdot n=0$ 
\item $n\bdot a\leq n$
\item If $a\leq b$, then $a\bdot b=0$
\item If $a\leq b$, then $a\bdot c\leq b\bdot c$
\item If $a\leq b$, then $n\bdot b\leq n\bdot a$\,.
\end{enumerate}(We note that the order $\leq$ being used is the natural order. We cannot use the BCK-order on $\D_n$ yet!)

Notice also that if none of $x$, $y$, or $z$ are equal to $n$, then (BCK1) is satisfied since we are in the subalgebra $\C_n$ which is known to be a BCK-algebra. This leaves seven cases to consider. If two or more of $x$, $y$, and $z$ are equal to $n$, it is a straightforward computation to show that (BCK1) is satisfied using the observations above.

If $y=n$, then \[\bigl[(x\bdot y)\bdot(x\bdot z)\bigr]\bdot(z\bdot y)=\bigl[(x\bdot n)\bdot(x\bdot z)\bigr]\bdot(z\bdot n)=\bigl[0\bdot(x\bdot z)\bigr]\bdot(z\bdot n)=0\bdot (z\bdot n)=0\,.\]

If $z=n$, then \[\bigl[(x\bdot y)\bdot(x\bdot n)\bigr]\bdot(n\bdot y)=\bigl[(x\bdot y)\bdot 0\bigr]\bdot(n\bdot y)=(x\bdot y)\bdot(n\bdot y)=0\,.\] The last equality follows by observations (3) and (4): since $x<n$, we have $x\bdot y<n\bdot y$, and thus $(x\bdot y)\bdot (n\bdot y)=0$.

The last case to consider is the case where $x=n$. If $y=0$, then (BCK1) reduces to (BCK2) which we proved above. If $z=0$, then 
\[\bigl[(n\bdot y)\bdot(n\bdot 0)\bigr]\bdot(0\bdot y)=\bigl[(n\bdot y)\bdot n\bigr]\bdot 0=(n\bdot y)\bdot n=0\] since $n\bdot y\leq n$. We may assume then that $0<y,z<n$. If $z\leq y$, then $n\bdot y\leq n\bdot z$ by observation (5), and so
\[\bigl[(n\bdot y)\bdot(n\bdot z)\bigr]\bdot(z\bdot y)=0\bdot (z\bdot y)=0\,.\] Finally, suppose $0<y<z<n$. Note that $y\in\{1,2\ldots, n-2\}$, so $n\bdot y=n-y-1$ and $z\bdot y=z-y$. If $z=n-1$, then 
\[\bigl[(n\bdot y)\bdot(n\bdot z)\bigr]\bdot(z\bdot y)=\bigl[(n-y-1)\bdot 1\bigr]\bdot(n-1-y)=(n-y-2)\bdot (n-y-1)=0\] since $n-y-2<n-y-1$. If $z\neq n-1$, then $n\bdot z=n-z-1$, in which case $(n\bdot y)\bdot (n\bdot z)=(n-y-1)\bdot (n-z-1)=(n-y-1)-(n-z-1)=z-y$, and thus we have
\[\bigl[(n\bdot y)\bdot(n\bdot z)\bigr]\bdot(z\bdot y)=(z-y)\bdot (z-y)=0\,.\]

Hence, $\D_n$ satisfies (BCK1) and $\D_n$ is a BCK-algebra.

Lastly, we check that $(n,n-1)$ is a non-commuting pair:
\begin{align*}
(n-1)\meet n &= n\bdot \bigl(n\bdot(n-1)\bigr)=n\bdot 1=n-2\\
n\meet (n-1)&=(n-1)\bdot \bigl((n-1)\bdot n\bigr)=(n-1)\bdot 0=n-1
\end{align*} and so $\D_n$ is non-commutative.
\end{proof}

\bibliographystyle{plain}
\bibliography{sat_deg_bib}

\end{document}